\documentclass{amsart}

\usepackage{xcolor}
\usepackage{amsmath,amssymb}
\usepackage{amscd}
\usepackage{mathrsfs}
\usepackage{amsthm}
\theoremstyle{plain}
\usepackage{mathtools, todonotes}

\newtheorem{theorem}{\bf Theorem}[section]
\newtheorem{lemma}[theorem]{\bf Lemma}
\newtheorem{proposition}[theorem]{\bf Proposition}

\newtheorem{conjecture}[theorem]{\bf Conjecture}

\theoremstyle{definition}
\newtheorem{definition}[theorem]{\bf Definition}
\newtheorem{example}[theorem]{\bf Example}

\newtheorem{remark}[theorem]{\bf Remark}

\newcommand{\R}{\mathbb{R}}

\newcommand{\N}{\mathbb{N}}

\newcommand{\disp}{\displaystyle}

\newcommand{\nai}[2]{\langle #1,#2\rangle}

\newcommand{\eqa}[1]{
\begin{align*}
#1
\end{align*}}

\newcommand{\IN}{\mathbb N}

\newcommand{\ve}{\varepsilon}

\newcommand{\ip}[1]{\mathopen{\langle}#1\mathclose{\rangle}}

\usepackage[all]{xy}

\title{Proof of the Paszkiewicz conjecture about a product of positive contractions}
\date{\today}

\author[H. Ando]{Hiroshi Ando}
\address{Hiroshi~Ando, Department of Mathematics and Informatics, Chiba University, 1-33 Yayoi-cho, Inage, Chiba, 263- 8522,
Japan}
\email{hiroando@math.s.chiba-u.ac.jp}

\author[Y.~Miyamoto]{Yuki Miyamoto}
\address{Yuki~Miyamoto, Department of Mathematics and Informatics, Chiba University, 1-33 Yayoi-cho, Inage, Chiba, 263- 8522,
Japan}
\email{23wm0119@student.gs.chiba-u.jp}

\author[N.~Ozawa]{Narutaka Ozawa}
\address{Narutaka~Ozawa, Research Institute for Mathematical Sciences, Kyoto University Kyoto, 606-8502, Japan}
\email{narutaka@kurims.kyoto-u.ac.jp}

\let\origmaketitle\maketitle
\def\maketitle{
  \begingroup
  \def\uppercasenonmath##1{} 
  \let\MakeUppercase\relax 
  \origmaketitle
  \endgroup
}
\begin{document}
\maketitle
\begin{abstract}
The Paszkiewicz conjecture about a product of positive contractions asserts that given a decreasing sequence $T_1\ge T_2\ge \dots$ of positive contractions on a separable infinite-dimensional Hilbert space, the product $S_n=T_n\dots T_1$ converges strongly. Recently, the first named author verified the conjecture for certain classes of sequences. In this paper, we prove the Paszkiewicz conjecture in full generality. Moreover, we show that in some cases, a generalized version of the Paszkiewicz conjecture also holds. 
\end{abstract}

\noindent

\medskip

\noindent
{\bf Mathematics Subject Classification (2020) 47A63, 47A65}.\\
Keywords: positive contractions, operator products, Hilbert spaces, spectral theory
\medskip

\tableofcontents 
\section{Introduction}
Let $H$ be a separable infinite-dimensional Hilbert space, which we fix throughout the paper. The Paszkiewicz's conjecture about a product of positive contraction is the following.   
\begin{conjecture}[Adam Paszkiewicz, 2018]\label{conj Paszkiewicz} 
Let $T_1\ge T_2\ge \dots $ be a sequence of positive linear contractions on $H$. 
Then the sequence $S_n:=T_nT_{n-1}\cdots T_1$ converges strongly. 
\end{conjecture}
Since $T_1\ge T_2\ge \dots$ is a decreasing sequence of positive contractions, the limit $T:=\lim_{n\to \infty}T_n$ (SOT) exists (SOT stands for the strong operator topology). 
We will use the notation that for a Borel subset $A$ of $\R$, $1_A(T)$ denotes the spectral projection of $T$ corresponding to $A$. Let $P:=1_{\{1\}}(T)$. 
In \cite{AMPaszkiewicz}, Conjecture \ref{conj Paszkiewicz} is shown to be equivalent to Conjecture \ref{conj strong Paszkiewicz} below. 
\begin{conjecture}\label{conj strong Paszkiewicz} Let $T_1\ge T_2\ge \dots $ be as in Conjecture \ref{conj Paszkiewicz}. Then
$\disp \lim_{n\to \infty}S_n=P$ ($*$-strongly). 
\end{conjecture}

The Conjecture \ref{conj strong Paszkiewicz} is easily seen to be true in the following cases: 
\begin{example}\label{ex easy case}
\begin{itemize}
\item[(1)] The constant sequence $T_n\equiv T\,(\forall n)$. Then $S_n=T^n\to P$ (SOT). 
\item[(2)] Each $T_n$ is a projection $P_n$. Then $T=\lim_nT_n$ is also a projection, say $P$, and $S_n=P_n\to P$ (SOT). 
\item[(3)] $\|T_{n_0}\|<1$ for some $n_0$. Then because $\|T_{n+1}\|\le \|T_n\|\le \cdots$, $\|S_{n_0+k-1}\|\le \|T_{n_0}\|^k\to 0\,(k\to \infty)$. Thus, $S_n\to P=0$ in norm.  
\item[(4)] $T_nT_m=T_mT_n$ for all $n,m\in \N$. In this case, if $X$ is the Gelfand spectrum of the unital abelian C$^*$-algebra generated by  $\{T_1,T_2,\dots\}$, then we may view $T_n=f_n$ for some $f_n\in C(X)$ and $f_1(x)\ge f_2(x)\ge \dots \ge f(x)=\lim_nf_n(x)$ for $x\in X$. Thus 
$S_n(x)=f_n(x)\cdots f_1(x)$. We may identify $H=L^2(X,\mu)$ for some Borel probability measure $\mu$ on $X$ with full support. Let $x\in X$. If $f(x)=1$, then $f_k(x)=1$ for every $k\in \N$ and thus $S_n(x)=1$ for every $n\in \N$. 
If $f(x)<1$, then there exists $k_0\in \N$ such that $f_k(x)<\frac{f(x)+1}{2}<1$ for every $k\ge k_0$, and thus $S_{n+k_0}(x)\le \left (\frac{f(x)+1}{2}\right )^n\to 0\,(n\to \infty)$. Thus, $S_n(x)\to 1_{f^{-1}(\{1\})}$ pointwise, whence $S_n\to P$ (SOT). 
\end{itemize}
\end{example}
Moreover, it is proved in \cite{AMPaszkiewicz} that $\disp \lim_{n\to \infty}S_n^*=P$ (SOT), and the Paszkiewicz conjecture is true if either (i) the von Neumann algebra $\mathscr{M}=W^*(T_1,T_2,\dots)$ generated by $T_1,T_2,\dots$ is finite, i.e., it admits a faithful normal tracial state, or (ii) $T_1,T_2,\dots$ has uniform spectral gap at 1, i.e., there exists $\delta\in (0,1)$ and $N\in \N$ such that $\sigma(T_n)\cap (1-\delta,1)=\emptyset$ holds for all $n\ge N$. However, these are restrictive classes of sequences. Indeed, any von Neumann algebra $\mathscr{M}$ on $H$ is of the form $\mathscr{M}=W^*(T_1,T_2,\dots)$ for some decreasing sequence of positive contractions $T_1\ge T_2\ge \dots$ (Proposition \ref{prop any M can be realized}). 
In this paper, we show that the Paszkiewicz conjecture is true in full generality. 

In $\S$\ref{sec proof of Paszkiewicz conjecture}, we give a proof of the Paszkiewicz conjecture. We add an example to show that nevertheless, a product of decreasing sequence of positive contractions can have similar behaviour to orthogonal transformations (Proposition \ref{prop 2-dim}). 

In $\S$\ref{sec generalized Paszkiewicz conjecture}, we consider a generalization of the Paszkiewicz conjecture. The motivation behind such a generalization came from our experience that many of our earlier arguments we have discovered so far to prove the original Paszkiewicz conjecture for some classes of sequences rely not too much on the fact that in the definition of the product $S_n=T_n\dots T_1$, the operators appear in the monotone decreasing order, though it is crucial that the $T_n$ converges to $T$. This leads us to consider the following generalization of the Paszkiewicz conjecture. 
We denote by $\mathscr{S}$ the set of all self-maps $\sigma\colon \N\to \N$ which are proper. 
For $n\in \N$ and $\sigma\in \mathscr{S}$, define 
\[S_n^{\sigma}=T_{\sigma(n)}\dots T_{\sigma(1)}.\]
Then one can show that the $\sigma$-Paszkiewicz subspace $\disp H_{\sigma}=\left \{\xi\in H\,\middle|\, \lim_{n\to \infty}S_n^{\sigma}\xi=P\xi\right \}$
is a closed subspace of $H$, and we set 
\[H_{\mathscr{S}}=\bigcap_{\sigma\in \mathscr{S}}H_{\sigma}.\]
We say that the generalized Paszkiewicz conjecture holds for $T_1\ge T_2\ge \dots$, if $H_{\mathscr{S}}=H$ holds. 
We show in Theorem \ref{thm HS0 central projection} that if $\mathscr{M}$ is a factor, then the generalized Paszkiewicz conjecture holds if and only if $H_{\mathscr{S}}\neq \{0\}$. 
We then remark that some classes of sequences do satisfy the generalized Paszkiewicz conjecture.  
In particular, we show that this is the case if either (1) $T_1$, $T_2,\dots$ has uniform spectral gap at 1 (Theorem \ref{prop unif gap strong P}), extending the work \cite{AMPaszkiewicz}, or (2) $\disp \lim_{n\to \infty}\|T_n-T\|=0$ and $1$ is isolated in the spectrum $\sigma(T)$ of $T$ (the latter is a special case of the former, but it leads to the norm convergence of $S_n$. See Remark \ref{rem norm convergence and  1 isolated in spec(T)}).     

\section{Proof of the Paszkiewicz conjecture}\label{sec proof of Paszkiewicz conjecture}
Here, we prove the Paszkiewicz conjecture (note that the seaparability of $H$ is not required). 

\begin{proof}[Proof of the Paszkiewicz conjecture]
For each $m\in \N$, let $P_m=1_{\{1\}}(T_m)$ and $P=1_{\{1\}}(T)$, with $\disp T=\lim_{n\to \infty}T_n$ (SOT). Then $P_1\ge P_2\ge \dots \ge P$ and $\disp \lim_{n\to \infty}P_n=P$ (SOT) by \cite[Lemma 2.2]{AMPaszkiewicz}. Since $T_nP=PT_n=P$, we have $S_nP=P$ for each $n\in \N$. Thus, it suffices to show that $\disp \lim_{n\to \infty}S_nP^{\perp}=0$ (SOT). By $P_n^{\perp}\nearrow P^{\perp}$ (SOT), it suffices to show that $S_nP_m^{\perp}\xrightarrow{n\to \infty}0$ (SOT) for every $m\in \N$.   

Let $m\in\IN$, $\xi_1\in P_m^{\perp}(H)$, and $\ve>0$ be given 
and set $\xi_n:=T_{n-1}\xi_{n-1}=T_{n-1}\cdots T_1\xi_1$ for $n\geq2$.
It suffices to show  
$\lim_n \| \xi_n\|^2 \le \ve$. 
Note that $\xi_m\in P_m^{\perp}(H)$ by $T_jP_m=P_mT_j=P_m\,(j\le m)$, and that the sequence 
$\|\xi_n\|$ is decreasing. 
Consider the positive increasing functions 
$f_k(t) := 1-(1-t)^{1/k}$ on $[0,1]$. 
Since $f_k$ converges to the characteristic function for $\{1\}$, 
one has 
\[
\lim_{k\to \infty} \ip{f_k(T_m)\xi_m,\xi_m} = \| 1_{\{1\}}(T_m)\xi_m\|^2=0.
\] 
Take $k$ such that $\ip{f_k(T_m)\xi_m,\xi_m} \le \ve/4$ and write $f:=f_k$. 
Since $f$ is operator monotone, one has 
\[
\ip{f(T_{n})\xi_{n},\xi_{n}}
 \le \ip{ T_{n-1} f(T_{n-1}) T_{n-1}\xi_{n-1},\xi_{n-1}}
 \le \ip{f(T_{n-1})\xi_{n-1},\xi_{n-1}} \le \cdots \le \ve/4
\]
for all $n\geq m$. 
Let $e_{T_n}$ be the spectral measure associated with $T_n$ and $\mu_n=\nai{e_{T_n}(\cdot)\xi_n}{\xi_n}$ the corresponding scalar-valued spectral measure 
for $(T_{n},\xi_{n})$. 
Fix $n\geq m$ for a moment and suppose that 
$\mu_n([0,1])=\|\xi_{n}\|^2\geq \ve$. 
Put $\gamma:=f^{-1}(1/2)\in(0,1)$.
Since $\int f\,d\mu_n\le \ve/4$, one has $\mu_n([\gamma,1])\le \ve/2$ 
and so $\mu_n([0,\gamma))\geq \ve/2$. 
It follows that 
\[
(1-\gamma^2)\ve/2 \le (1-\gamma^2)\mu_n([0,\gamma)) 
\le \int 1-t^2 \,d\mu_n(t) = \|\xi_n\|^2-\|\xi_{n+1}\|^2 
\]
as long as $\|\xi_n\|^2\geq\ve$. 
This implies $\|\xi_n\|^2 \le \ve$ 
for any $n\geq m+2(1-\gamma^2)^{-1}\ve^{-1}\|\xi_m\|^2$. 
\end{proof}

We complement the proof with the following example.

\begin{proposition}\label{prop 2-dim}
For every $0<\delta<1$, there are $n\in\IN$ and a decreasing sequence 
$I\geq T_1\geq T_2 \geq \cdots \geq T_n$
of positive contractions on the $2$-dimensional Hilbert space 
$\ell_2^2$ that satisfy $T_n\geq (1-\delta)I$ and 
\[
T_n\cdots T_1\left(\begin{smallmatrix} 1 \\ 0\end{smallmatrix}\right)
=\left(\begin{smallmatrix} 0 \\ 1-\delta\end{smallmatrix}\right)
\]
\end{proposition}

We need a preparatory construction. 
For a vector $\xi$, we denote by $P_\xi$ 
the rank-one orthogonal projection associated with $\xi$ 
and set $P_\xi^\perp:=I-P_\xi$. 
For a vector $\xi$, we denote by $P_\xi$ 
the rank-one orthogonal projection associated with $\xi$ 
and set $P_\xi^\perp:=I-P_\xi$. 

\begin{lemma}\label{alpha}
Let $\ve\in[0,1/2]$, $\theta\in[0,\pi/3]$, 
$\xi:=\left(\begin{smallmatrix} 1 \\ 0 \end{smallmatrix}\right)$, and 
$\eta:=\left(\begin{smallmatrix} \cos\theta \\ \sin\theta\end{smallmatrix}\right)$.
Let 
$\phi\in[0,\pi/3]$, 
$\zeta=\left(\begin{smallmatrix} \cos\phi \\ \sin\phi \end{smallmatrix}\right)$, 
and $\alpha\in[0,1]$ be given by the equation   
\[
(P_\eta + (1-\ve) P_\eta^\perp)\xi = \alpha \zeta.
\]
Then, one has $\alpha\geq1-\ve\theta^2$ 
and $|\phi-\ve\theta|\le C\ve\theta^3$, where $C$ is an absolute constant.
\end{lemma}
\begin{proof}
Write $c:=\cos\theta$ and $s:=\sin\theta\le\theta$. 
Set $\eta':=\left(\begin{smallmatrix} s \\ -c\end{smallmatrix}\right)$. Then 
$\xi=c \eta + s \eta'$ and 
\[
(P_\eta+(1-\ve)P_\eta^\perp)\xi
 = c \eta + (1-\ve)s \eta'
\]
Hence 
\begin{align*}
\alpha =(1-2\ve s^2+\ve^2s^2)^{1/2}
 &=((1-\ve s^2)^2+\ve^2s^2(1-s^2))^{1/2}\\
 &=(1-\ve s^2)(1+t)^{1/2} \geq 1-\ve \theta^2,
\end{align*}
where $t:=\ve^2s^2(1-s^2)/(1-\ve s^2)^2=\ve^2 s^2 + O(\ve^2 s^4)$. 
This proves the first assertion. 
Since $(1+t)^{1/2}=1+t/2+O(t^2)$ and $(1-u)^{-1}=1+u+O(u^2)$, 
\[
\alpha^{-1} = 1 + \ve s^2 - \frac{1}{2}\ve^2 s^2 + O(\ve^2 s^4).
\]
Thus
\[
\cos\phi = \alpha^{-1}(1-\ve s^2) = 1-\frac{1}{2}\ve^2s^2 + O(\ve^2 s^4)
 = 1-\frac{1}{2}\ve^2\theta^2 + O(\ve^2 \theta^4),
\]
by $s=\theta+O(\theta^3)$. 
Since $0\le \arccos(1-x) - (2x)^{1/2} \le x^{3/2}$ for $x\in[0,1]$, one has
\[
\phi=\ve\theta + O(\ve\theta^3).
\]
This proves the second assertion (in fact for $C = 10$ if the details are worked out).
\end{proof}

\begin{lemma}\label{beta}
Let $0<\ve<\ve+\ve^2<\kappa<1/2$, 
$\xi:=\left(\begin{smallmatrix} 1 \\ 0\end{smallmatrix}\right)$,  
and $\zeta:=\left(\begin{smallmatrix} \cos\phi \\ \sin\phi\end{smallmatrix}\right)$. 
Let $\beta\ge 0$ be the largest constant that satisfies 
\begin{equation}
P_\xi+(1-\ve)P_\xi^\perp
 \geq \beta (P_\zeta+(1-\kappa)P_\zeta^\perp).\label{eq betaP}
\end{equation}
Then  $1-2\phi^2\leq \beta \leq 1$. 
\end{lemma}
\begin{proof}
Note that (\ref{eq betaP}) implies that $\beta P_{\zeta}$ is a positive contraction, whence $\beta \leq 1$ holds.   
We may exchange $\xi$ and $\zeta$. 

Write $c:=\cos\phi$ and $s:=\sin\phi$.
Thus, $\beta$ is the smaller of the eigenvalues of 
\begin{align*}
\left(\begin{smallmatrix} 1 & 0\\ 0 & (1-\kappa)^{-1/2}\end{smallmatrix}\right)
\left(\begin{smallmatrix} 1 - \ve s^2 & \ve cs \\ \ve cs & 1-\ve c^2\end{smallmatrix}\right)
\left(\begin{smallmatrix} 1 & 0\\ 0 & (1-\kappa)^{-1/2}\end{smallmatrix}\right)
 &= \left(\begin{smallmatrix} 1 - \ve s^2 & (1-\kappa)^{-1/2}\ve cs \\ (1-\kappa)^{-1/2}\ve cs & 1+(1-\kappa)^{-1}(\kappa-\ve+\ve s^2) \end{smallmatrix}\right).
\end{align*}
Since the eigenvalues of 
$\left(\begin{smallmatrix} a & b \\ b & d \end{smallmatrix}\right)$ 
are $((a+d) \pm ((a-d)^2+4b^2)^{1/2})/2$, we estimate for the above matrix 
\[
a+d = 2 +\frac{\kappa-\ve}{1-\kappa}+\frac{\kappa\ve s^2}{1-\kappa} 
 \ge 2 +\frac{\kappa-\ve}{1-\kappa}
\]
and, since 
$d-a=(1-\kappa)^{-1}(\kappa-\ve + (2-\kappa) \ve s^2)
 \geq(1-\kappa)^{-1}(\kappa-\ve)$,  
\begin{align*}
((a-d)^2+4b^2)^{1/2}
 &\le d-a + \frac{2b^2}{d-a}\\
 &\le \frac{\kappa-\ve}{1-\kappa} + \frac{(2-\kappa) \ve s^2}{1-\kappa}+ \frac{2\ve^2 c^2s^2}{\kappa-\ve}\\
 &\le \frac{\kappa-\ve}{1-\kappa}+4s^2.
\end{align*}
Thus $\beta\geq1-2s^2$. 
\end{proof}

\begin{proof}[Proof of Proposition~\ref{prop 2-dim}]
We may assume that $D:=2\delta^{-1}>8$ and 
set $\ve_k:=(D+n-k)^{-1}$ 
and $\theta_k:=(\log(D+n-k))^{-1}\theta$ for $k=0,1,\ldots,n-1$, 
where $n\in\IN$ and $\theta>0$ are chosen later. 
Note that $\ve_n\le \delta/2$.
We recursively define 
$\xi_k:=\left(\begin{smallmatrix} \cos\phi_k \\ \sin\phi_k\end{smallmatrix}\right)$, 
$\eta_k:=\left(\begin{smallmatrix} \cos\rho_k \\ \sin\rho_k\end{smallmatrix}\right)$, 
and $\alpha_k$ as follows. 
First, set $\phi_0:=0$ and $\rho_0:=\theta_0$. 
For $k=0,\ldots,n-1$,
set $\phi_{k+1}$ and $\alpha_{k+1}$ by the equation 
\[
(P_{\eta_k}+(1-\ve_k)P_{\eta_k}^\perp)\xi_k = \alpha_{k+1}\xi_{k+1}.
\]
Then, $\rho_{k+1}:=\phi_{k+1}+\theta_{k+1}$. 
Note that, by Lemma~\ref{alpha}, 
one has 
\[
\alpha_{k+1}\geq1-\ve_k\theta_k^2
\]
and
\[
|\phi_{k+1}-\phi_k - \ve_k\theta_k| \le C\ve_k\theta_k^3.
\]
Note also that we may assume (by taking $\theta$ sufficiently small) that $C\ve_k\theta_k^3\ll\ve_k\theta_k$, and in particular that $C\varepsilon_k\theta_k^3\le \frac{1}{2}\varepsilon_k\theta_k$ and $1-18\varepsilon_k\theta_k^2\ge \frac{1}{2}$ for all $k=0,\dots,n-1$.
We recall that $\log(\log(t))'=(t\log(t))^{-1}$. 
Thus, 
\[
\phi_n \approx \sum_{k=0}^{n-1}\ve_k\theta_k 
\approx \theta \log(\log n)
\]
as $n\to\infty$. 
Since $\phi_n$ depends continuously on $\theta$, by the intermediate value theorem, we may
set $n$ and $\theta$ to be such that $\phi_n=\pi/2$, 
i.e., $\xi_n=\left(\begin{smallmatrix} 0 \\ 1\end{smallmatrix}\right)$.
We use the inequality $t^{-1}(\log t)^{-2}\le \frac{3}{2}(t+1)^{-1}(\log(t+1))^{-2}, (t\ge 8)$ to obtain 
\eqa{
0&\le \theta_{k+1}-\theta_{k} \le (D+n-k-1)^{-1}(\log(D+n-k-1))^{-2}\theta\\\\
&\le \tfrac{3}{2}(D+n-k)^{-1}(\log (D+n-k))^{-2}\theta\le \tfrac{3}{2}\ve_k\theta_k.
}
Thus, one has 
\eqa{
0&\le\rho_{k+1}-\rho_k=\theta_{k+1}-\theta_k+\phi_{k+1}-\phi_k\\
 &\le \tfrac{3}{2}\ve_k\theta_k+\ve_k\theta_k+C\ve_k\theta_k^3\\
 &\le 3\ve_k\theta_k.
}

Set $\beta_k>0$ to be the largest constant that satisfies
\[
(P_{\eta_{k-1}}+(1-\ve_{k-1})P_{\eta_{k-1}}^\perp)\geq\beta_k(P_{\eta_k}+(1-\ve_k)P_{\eta_k}^\perp).
\]
Since 
$\ve_{k}-\ve_{k-1}=(D+n-k)^{-1}(D+1+n-k)^{-1}>\ve_{k-1}^2$, 
Lemma~\ref{beta} implies
\[
\beta_k \geq 1-2(\rho_k-\rho_{k-1})^2
 \geq 1-
18\ve_{k-1}^2\theta_{k-1}^2.
\]
Finally, set $\beta_0:=1$ and 
\[
T_k:=\beta_0\beta_1\cdots\beta_k(P_{\eta_k}+(1-\ve_k)P_{\eta_k}^\perp)
\le T_{k-1}.
\]
Then, one has 
\[
T_{n-1}\cdots T_0 \xi_0
 =(\prod_{k=1}^{n}\alpha_k)(\prod_{k=1}^{n-1}\beta_k^{n-k})\xi_n.
\]
For $B:=\int_D^\infty (t(\log t)^2)^{-1}\,\mathrm{d}t=(\log D)^{-1}$,
\[
\sum_{k=1}^n (1-\alpha_k) \le \sum_{k=0}^{n-1}\ve_k\theta_k^2
\le B\theta^2 
\]
and by $\varepsilon_k(n-k)\le 1$,
\[
\sum_{k=1}^{n-1}(1-\beta_k)(n-k)\le 18\sum_{k=0}^{n-2}\ve_k^2\theta_k^2(n-k)
\le 18\sum_{k=0}^{n-1}\varepsilon_k\theta_k^2\le 18B\theta^2. 
\]
Since we arranged $\theta$ sufficiently small so that $0\le 1-\alpha_{k}\le \frac{1}{2}$ and $0\le 1-\beta_k\le \frac{1}{2}$ for all $k=1,\dots,n$ we have 
$\log (\alpha_k)=\log (1-(1-\alpha_k))\ge -2(1-\alpha_k)$ and $\log(\beta_k)\ge -2(1-\beta_k)$ thanks to the inequality $\log (1-x)\ge -2x\,(0\le x\le \frac{1}{2})$. 
Therefore, 
\eqa{
\sum_{k=1}^n\log \alpha_k&\ge \sum_{k=1}^n-2(1-\alpha_k)\ge -2B\theta^2, one has
}
and similarly 
\eqa{
\sum_{k=1}^{n-1}\log(\beta_k^{n-k})\ge -2\sum_{k=1}^{n-1}(n-k)(1-\beta_k)\ge -32B\theta^2.
}

These imply that 
$(\prod_{k=1}^{n}\alpha_k)(\prod_{k=1}^{n-1}\beta_k^{n-k})\geq
\exp(-34B\theta^2)$. 

Since $\theta>0$ could have been arbitrarily small, we may have $\exp(-34B\theta)>1-\delta$. Then replacing $T_n$ with $\lambda T_n$, where $\lambda\in (0,1)$ is chosen such that $\lambda (\prod_{k=1}^{n}\alpha_k)(\prod_{k=1}^{n-1}\beta_k^{n-k})=1-\delta$, we are done. 
\end{proof}

We end this section with the following remarks. 
First, if $x_1,x_2,\dots$ is a sequence of (not necessarily positive) contractions on $H$, then we have a decreasing sequence $T_1\ge T_{2}\ge \dots$ of positive contractions, where 
\[T_n=y_n^*y_n,\,\,y_n=x_n\cdots x_1,\,\,n\in \N.\]
Actually, any $T_1\ge T_2\ge \dots$ is of this form:
\begin{proposition}\label{prop T_n as a product of x_n}
Let $T_1\ge T_2\ge \cdots$ be a decreasing sequence of positive contractions on $H$. Then there exist a sequence $(x_n)_{n=1}^{\infty}$ of contractions such that 
\[T_n=y_n^*y_n,\,\,y_n=x_n\cdots x_1,\,\,n\in \N.\]
Moreover, $W^*(T_1,T_2,\dots)=W^*(x_1,x_2,\dots)$ holds. 
\end{proposition}
The following lemma is well-known. We include the proof for completeness. 
\begin{lemma}\label{lem S=xT}Let $S,T\in \mathbb{B}(H)$ be such that $0\le S\le T$. 
Then there exists a contraction $x\in W^*(S,T)$ such that $S^{\frac{1}{2}}=xT^{\frac{1}{2}}$. 
\end{lemma}
\begin{proof} 
Set $\mathscr{M}=W^*(S,T)=W^*(S^{\frac{1}{2}},T^{\frac{1}{2}})$. 
Consider the orthogonal decomposition $H=\overline{\rm ran}(T^{\frac{1}{2}})\oplus \ker{T^{\frac{1}{2}}}$. We first define $x_0\colon {\rm ran}(T^{\frac{1}{2}})\to H$ by $x_0(T^{\frac{1}{2}}\xi)=S^{\frac{1}{2}}\xi$ for $\xi\in H$. Note that by $0\le S\le T$, we have $\|S^{\frac{1}{2}}\xi\|\le \|T^{\frac{1}{2}}\xi\|$, so that $x_0$ is a well-defined contraction. Thus, it extends to a contraction from $\overline{\rm ran}(T^{\frac{1}{2}})$ to $H$, still denoted by $x_0$. We then define $x$ by setting $x_0$ on $\overline{\rm ran}(T^{\frac{1}{2}})$ and 0 on $\ker(T^{\frac{1}{2}})$. Then $S^{\frac{1}{2}}=xT^{\frac{1}{2}}$ holds. 
We show that $x\in \mathscr{M}$. Let $y'$ be an element in the commutant $\mathscr{M}'$ of $\mathscr{M}$. For each $\xi\in H$, we have
$y'x_0(T^{\frac{1}{2}}\xi)=y'S^{\frac{1}{2}}\xi=S^{\frac{1}{2}}y'\xi=x_0(T^{\frac{1}{2}}y'\xi)=x_0y'(T^{\frac{1}{2}}\xi)$. 
Therefore, $y'x_0=x_0y'$ on $\overline{\rm ran}(T^{\frac{1}{2}})$. On the other hand, if $\xi\in \ker(T^{\frac{1}{2}})$, then $y'\xi\in \ker(T^{\frac{1}{2}})$ by $T^{\frac{1}{2}}y'=y'T^{\frac{1}{2}}$, whence $y'x\xi=xy'\xi=0$. This shows that $y'x=y'x$. Therefore, we obtain $x\in \mathscr{M}''=\mathscr{M}$. 
\end{proof}
\begin{proof}[Proof of Proposition \ref{prop T_n as a product of x_n}]
Set $x_1=T_1^{\frac{1}{2}}$. 
For each $n\in \N$, we may apply Lemma \ref{lem S=xT} to $0\le T_{n+1}\le T_n$ to find a contraction $x_{n+1}\in W^*(T_{n+1},T_n)$ such that $T_{n+1}^{\frac{1}{2}}=x_{n+1}T_n^{\frac{1}{2}}$. 
Then $T_1=x_1^*x_1$, and for $n\ge 2$, 
\eqa{T_{n}&=T_{n-1}^{\frac{1}{2}}x_{n}^*x_{n}T_{n-1}^{\frac{1}{2}}\\
&=T_{n-2}^{\frac{1}{2}}x_{n-1}^*x_{n}^*x_{n}x_{n-1}T_{n-2}^{\frac{1}{2}}\\
&=\cdots =x_1^*\cdots x_{n}^*x_{n}\cdots x_1\\
&=y_{n}^*y_{n},
}
where $y_n=x_n\cdots x_1\,\,(n\in \N)$. Finally, since $T_n=y_n^*y_n\in W^*(x_1,x_2,\dots)\,(n\in \N)$, $W^*(T_1,T_2,\dots)\subset W^*(x_1,x_2,\dots)$ holds.  
On the other hand, $x_1=T_1^{\frac{1}{2}}\in W^*(T_1,T_2,\dots)$
and $x_{n+1}\in W^*(T_{n+1}^{\frac{1}{2}},T_n^{\frac{1}{2}})\subset W^*(T_1,T_2,\dots)\,(n\in \N)$ implies that $W^*(x_1,x_2,\dots)\subset W^*(T_1,T_2,\dots)$. Therefore, $W^*(T_1,T_2,\dots)=W^*(x_1,x_2,\dots)$ holds. 
\end{proof}

Next, we show that any von Neumann algebra arises in the form of $W^*(T_1,T_2,\dots)$. Thus, the finiteness assumption of $W^*(T_1,T_2,\dots)$ in \cite{AMPaszkiewicz} is indeed quite restrictive. 
\begin{proposition}\label{prop any M can be realized}
Let $\mathscr{M}$ be a von Neumann algebra on $H$. Then there exist positive contractions $T_1\ge T_2\ge \dots$ such that $\mathscr{M}=W^*(T_1,T_2,\dots)$ holds. 
\end{proposition}
\begin{proof}
Since $H$ is separable, there exists a countable family of projections $e_1,e_2,\dots$ in $\mathbb{B}(H)$ such that $\mathscr{M}=W^*(e_1,e_2,\dots)$. Then $x_n=\tfrac{1}{2}(e_n+1),\,(n\in \N)$ is an invertible contraction and $W^*(x_1,x_2,\dots)=W^*(e_1,e_2,\dots)=\mathscr{M}$. Let $T_n=y_n^*y_n$, $y_n=x_n\cdots x_1$. Then $T_1\ge T_2\ge \cdots$ is a decreasing sequence of positive contractions on $H$ and $\tilde{\mathscr{M}}:=W^*(T_1,T_2,\dots)\subset \mathscr{M}$ holds. Since all $x_1,x_2,\cdots$ are invertible, so are $y_1,y_2,\dots$. We have $x_1=T_1^{\frac{1}{2}}\in \tilde{\mathscr{M}}$. Assume we have shown that $x_1,\dots x_n\in \tilde{\mathscr{M}}$. 
Then $y_1,\dots,y_n\in \tilde{\mathscr{M}}$, whence $T_{n+1}=y_n^
*x_{n+1}^2y_n$ implies $x_{n+1}=((y_n^*)^{-1}T_{n+1}y_n^{-1})^{\frac{1}{2}}\in \tilde{\mathscr{M}}$. This shows that $\mathscr{M}\subset \tilde{\mathscr{M}}$.  
Therefore, $\tilde{\mathscr{M}}=\mathscr{M}$ holds. 
\end{proof}

\section{Generalization of the Paszkiewicz conjecture}\label{sec generalized Paszkiewicz conjecture}
 In this section, we consider a generalization of the Paszkiewicz conjecture. 
 
\begin{definition}
A map $\sigma\colon \N\to \N$ is called proper, if for every $k\in \N$, the set $\sigma^{-1}(\{k\})$ is finite. The set of all proper maps from $\N$ to itself is denoted by $\mathscr{S}$. 
\end{definition}
\begin{remark}
A map $\sigma\colon \N\to \N$ is proper if and only if $\disp \lim_{n\to \infty}\sigma(n)=\infty$, i.e., for every $N\in \N$, there exists $n_0\in \N$ such that $\sigma(n)>N$ for every $n\ge n_0$. 
\end{remark} 
Now for a proper map $\sigma\colon \N\to \N$, we set
\[S_n^{\sigma}=T_{\sigma(n)}\dots T_{\sigma(1)}.\]
 
By $\disp \lim_{n\to \infty}\sigma(n)=\infty$, $\disp \lim_{n\to \infty}T_{\sigma(n)}=T$ (SOT) holds.
Then we have the following analogue of \cite[Proposition 2.3]{AMPaszkiewicz}. The proof is essentially the same, so we do not repeat it here.  

\begin{proposition}\label{prop: S_nsigma*converges}Let $T_1\ge T_2\ge \cdots$ be a sequence of positive contractions on $H$, and let $S_n^{\sigma}:=T_{\sigma(n)}\cdots T_{\sigma(1)}$. 
    The following statements hold (WOT stands for the weak operator topology): 
    \begin{list}{}{}
    \item[{\rm{(1)}}] $\disp \lim_{n\to \infty}(S_n^{\sigma})^*=P$ {\rm (SOT)}. In particular, $\disp \lim_{n\to \infty}S_n^{\sigma}=P$ {\rm (WOT)} holds.
    \item[{\rm{(2)}}] Let $\xi\in H$. If the set $\{S_n^{\sigma}\xi\mid\,n\in \mathbb{N}\}$ is totally bounded, then $\disp \lim_{n\to \infty}\|S_n^{\sigma}\xi-P\xi\|=0$ holds. 
    \item[{\rm{(3)}}] For every $\xi\in H$ and every $k\in \mathbb{N}$, $\displaystyle \lim_{n\to \infty}\|S_{n+k}^{\sigma}\xi-S_n^{\sigma}\xi\|=0$ holds. 
    \end{list}
    \end{proposition}
Next, we introduce the $\sigma$-Paszkiewicz subspace $H_{\sigma}$. 
\begin{proposition}\label{prop Psubspace}
Let $T_1\ge T_2\ge \dots $ be positive contractions on $H$ and let $\sigma\in \mathscr{S}$. Then the set 
$$H_{\sigma}:=\left \{\xi \in H\,\middle|\, \lim_{n\to \infty}\|S_n^{\sigma}\xi-P\xi\|=0\right \}=\left \{\xi\in H\,\middle|\, \lim_{n\to \infty}\|S_n^{\sigma}\xi\|=\|P\xi\|\right \}$$
is a closed subspace of $H$ containing $P(H)$. 
\end{proposition}
\begin{proof}
It is clear that $H_{\sigma}$ is a vector subspace of $H$ and that $P(H)\subset H_{\sigma}$ because $S_n^{\sigma}P\xi=P\xi$ for every $\xi\in H$ and $n\in \N$. Let $\xi\in \overline{H_{\sigma}}$ and $\varepsilon>0$. Then there exists $\xi_0\in H_{\sigma}$ such that $\|\xi-\xi_0\|<\frac{\varepsilon}{3}$ holds. Since $\xi_0\in H_{\sigma}$, we may find an $n_0\in \N$ such that $\|S_m^{\sigma}\xi_0-S_n^{\sigma}\xi_0\|<\frac{\varepsilon}{3}$ for every $n,m\ge n_0$. Then for every $n,m\ge n_0$, 
\eqa{
\|S_m^{\sigma}\xi-S_n^{\sigma}\xi \|&\le \|S_m^{\sigma}(\xi-\xi_0)\|+\|S_m^{\sigma}\xi_0-S_n^{\sigma}\xi_0\|+\|S_n^{\sigma}\xi_0-S_n^{\sigma}\xi\|\\
&<2\|\xi-\xi_0\|+\frac{\varepsilon}{3}<\varepsilon.
}
Since $\varepsilon$ is arbitrary, this shows that $(S_n^{\sigma}\xi)_{n=1}^{\infty}$ is Cauchy, whence it converges to $P\xi$ by Proposition \ref{prop: S_nsigma*converges} (2). Therefore, $\xi\in H_{\sigma}$ holds. This shows that $H_{\sigma}$ is closed. Finally, to show the last equality, we show that for each $\xi\in H$ the following conditions are equivalent:
\begin{itemize}
\item[(a)] $\disp \lim_{n\to \infty}\|S_n^{\sigma}\xi-P\xi\|=0.$
\item[(b)] $\disp \lim_{n\to \infty}\|S_n^{\sigma}\xi\|=\|P\xi\|.$
\end{itemize}
(a)$\implies$(b) is clear. Assume (b). Then because $\disp \lim_{n\to \infty}S_n^{\sigma}=P$ (WOT), we have 
\eqa{
 \|S_n^{\sigma}\xi-P\xi\|^2&=\|S_n^{\sigma}\xi\|^2+\|P\xi\|^2-2\,{\rm Re}\nai{S_n^{\sigma}\xi}{P\xi}\\
 &\xrightarrow{n\to \infty}\|P\xi\|^2+\|P\xi\|^2-2\,{\rm Re}\nai{P\xi}{P\xi}=0.
}
Therefore, (a) holds. This finishes the proof. 
\end{proof}

Then, inspired by the role played by the subspace ``$Z$" in the work of Kopeck\'{a}--Paszkiewicz \cite[Lemma 3.4]{MR3642022KopeckaPaszkiewicz2017}, we define the following closed subspace $H_{\mathscr{S}}$ of $H$:
\begin{definition}
We define the $\sigma$-Paszkiewicz subspace $H_{\sigma}$  by
\[H_{\sigma}=\{\xi\in H\mid \lim_{n\to \infty}S_n^{\sigma}\xi=P\xi\},\]
and then we set 
\[H_{\mathscr{S}}=\bigcap_{\sigma\in \mathscr{S}}H_{\sigma}.\]
\end{definition}
Let $\mathscr{M}=W^*(T_1,T_2,\dots)$. 

\begin{definition}
We say that the generalized Paszkiewicz conjecture holds for $T_1\ge T_2\ge \dots$ if $H_{\mathscr{S}}=H$ holds. 
\end{definition}

\begin{theorem}\label{thm HS0 central projection}
The following statements hold. 
\begin{list}{}{}
\item[{\rm (1)}] Both $H_{\mathscr{S}}$ is a  closed subspace of $H$ which are invariant under all $T_n$. \item[{\rm (2)}] $e$ is a central projection in $\mathscr{M}$. In particular, if $\mathscr{M}$ is a factor, then the generalized Paszkiewicz conjecture holds for $T_1\ge T_2\ge \dots$ if and only if $H_{\mathscr{S}}\neq \{0\}$. 
\end{list}
\end{theorem}
\begin{proof}
(1) Let $\sigma\in \mathscr{S}$, $m,n\in \N$ and $\xi\in H_{\mathscr{S}}$. Then the map $\hat{\sigma}\colon \N\to \N$ defined by $\hat{\sigma}(1)=m, \hat{\sigma}(k+1)=\sigma(k),\,k\in \N$ is an element in $\mathscr{S}$. Thus, 
\[S_{n}^{\sigma}(T_m\xi)=S_{n+1}^{\hat{\sigma}}\xi\xrightarrow{n\to \infty}P\xi=P(T_m\xi)\]
by $\xi\in H_{\hat{\sigma}}$. Therefore $T_m\xi\in H_{\mathscr{S}}$. This shows that $H_{\mathscr{S}}$ is invariant under $T_m$.\\
(2) By (1), $e$ belongs to $\mathscr{M}'$. 
Assume that $f$ is a projection in $\mathscr{M}'$ such that $e\sim f$ in $\mathscr{M}'$ ($\sim$ denotes the Murray-von Neumann equivalence of projections). By $e\sim f$, there exists a partial isometry $u\in \mathscr{M}'$ such that $u^*u=e$  and $uu^*=f$. Let $\xi\in f(H)$ and $\sigma\in \mathscr{S}$. 
Then $u^*\xi\in e(H)\subset H_{\sigma}$. By $P,S_n^{\sigma}\in \mathscr{M}$ and $u\in \mathscr{M}'$, we have 
\[S_n^{\sigma}\xi=S_n^{\sigma}uu^*\xi=uS_n^{\sigma}(u^*\xi)\xrightarrow{n\to \infty}uPu^*\xi=Puu^*\xi=P\xi.\]
This shows that $\xi\in H_{\sigma}$. Since $\sigma\in \mathscr{S}$ is arbitrary, we obtain $\xi\in H_{\mathscr{S}}=e(H)$. Therefore, $f\le e$ holds. It then follows that 
\[
e\le \bigvee \left \{ueu^*\mid u\in U(\mathscr{M}')\right \}\le \bigvee \left \{f\in {\rm Proj}(\mathscr{M}')\mid f\sim e\right \}\le e,\] 
whence $e=\bigvee \left \{ueu^*\mid u\in U(\mathscr{M}')\right \}=z(e)$ holds (here, $z(e)$ is the central support of $e$). The last claim is then immediate. 
\end{proof}

We then give examples of sequences for which the  generalized Paszkiewicz conjecture holds (that many of such examples exist was the motivation for this generalization).  
It is clear that the generalized Paszkiewicz conjecture is true when $\mathscr{M}=W^*(T_1,T_2,\dots)$ is a finite von Neumann algebra, because $\displaystyle \lim_{n\to \infty}(S_n^{\sigma})^*=P$ (SOT) for every $\sigma\in \mathscr{S}$ and the $*$-operation is SOT-continuous on the unit ball of $\mathscr{M}$. 
The next result shows that the generalized Paszkiewicz conjecture is true also for sequences with uniform spectral gap at 1, which generalizes \cite{AMPaszkiewicz}.  
\begin{proposition}\label{prop unif gap strong P}
If $T_1\ge T_2\ge \dots$ has uniform spectral gap at 1, then the generalized Paszkiewicz conjecture holds for it. 
\end{proposition}
\begin{lemma}\label{lem TiTj}
Let $T_1\ge T_2$ be positive contractions on $H$. Let $P_i=1_{\{1\}}(T_i)\,(i=1,2)$. Then for each $i,j\in \{1,2\}$, $T_iP_j^{\perp}(H)\subset P_{\max\{i,j\}}^{\perp}(H)$ holds. 
\end{lemma}
\begin{proof}
The case $i=j$ is trivial. Assume $i<j$. Then $T_i\ge T_j$, whence $P_i\ge P_j$ and $T_iP_j=P_jT_i=P_j$. Thus $T_iP_j^{\perp}(H)=P_j^{\perp}T_i(H)\subset P_j^{\perp}(H)$. If $i>j$, then $P_j\ge P_i$, whence $P_j^{\perp}\le P_i^{\perp}$. Therefore, $T_iP_j^{\perp}(H)\subset T_iP_i^{\perp}(H)\subset P_i^{\perp}(H)$. 
\end{proof}
\begin{proof}[Proof of Proposition \ref{prop unif gap strong P}]
Let $\sigma\in \mathscr{S}$. 
Let $\varepsilon>0$ and $\xi\in P^{\perp}(H)$ be a unit vector. Then there exists $n_0\in \N$ such that $\|\xi-P_{n_0}^{\perp}\xi\|<\frac{1}{2}\varepsilon$ holds. 
By assumption, there exists $N\in \N$ and $\delta\in (0,1)$ such that $\sigma(T_n)\cap (1-\delta,1)=\emptyset$ for every $n\ge N$. We may assume that $N\ge n_0$. 
Using the properness of $\sigma$, we may find an increasing sequence of natural numbers 
\[2\le n_1<n_2<\cdots\]
such that $\sigma(n_k)>\max \{N,\sigma(1),\dots,\sigma(n_k-1)\}$
for every $k\in \N$. Indeed, by the properness of $\sigma$, there exists $n_1\ge 2$ such that $\sigma(n_1)\ge N$ holds. Assume that we have found $n_1<n_2<\cdots <n_k$. Then by the properness of $\sigma$, there exists $n\in \N$ for which $n>n_k$ and $\sigma(n)>\sigma(n_k)$ holds. Let $n_{k+1}$ be the smallest such $n$. Then $\sigma(n_{k+1})>N,\sigma(n_k)$, and if $n<n_k$, then $\sigma(n)<\sigma(n_k)<\sigma(n_{k+1})$ and if $n_k\le n<n_{k+1}$, then $\sigma(n_{k+1})>\sigma(n_k)\ge \sigma(n)$ by the choice of $n_{k+1}$. By induction, we have the $n_1<n_2<\dots$ with the required properties. 
Choose $k\in \N$ for which $(1-\delta)^k<\frac{1}{2}\varepsilon$ holds. 
Then for every $n\ge n_k$, we have 
\[\|S_n^{\sigma}\xi\|\le \|S_{n_k}^{\sigma}\xi\|\le \|S_{n_k}^{\sigma}(\xi-P_{n_0}^{\perp}\xi)\|+\|S_{n_k}^{\sigma}P_{n_0}^{\perp}\xi\|\]
and by a repeated use of Lemma \ref{lem TiTj}, we have
\eqa{
S_{n_k}^{\sigma}P_{n_0}^{\perp}\xi=T_{n_k}(\underbrace{T_{\sigma(n_k-1)}\cdots T_{\sigma(1)}P_{n_0}^{\perp}\xi)}_{\in P_{\max\{\sigma(1),\dots,\sigma(n_k-1),n_0\}}^{\perp}(H)}
}
And by $\sigma(n_k)>\max(\sigma(1),\dots,\sigma(n_k-1),n_0,N)$, we have 
\[P_{\max\{\sigma(1),\dots,\sigma(n_k-1),n_0\}}^{\perp}(H)\subset P_{\sigma(n_k)}^{\perp}(H)=1_{[0,1-\delta]}(T_{\sigma(n_k)})(H).\]
This implies that 
\eqa{
\|S_{n_k}^{\sigma}P_{n_0}^{\perp}\xi\|&\le (1-\delta)\|T_{\sigma(n_k-1)}\cdots T_{\sigma(n_{k-1}+1)}S_{n_{k-1}}^{\sigma}P_{n_0}^{\perp}\xi\|\\
&\le (1-\delta)\|S_{n_{k-1}}^{\sigma}P_{n_0}^{\perp}\xi\|\\
&\le \dots \le (1-\delta)^k\|S_{n_1-1}^{\sigma}P_{n_0}^{\perp}\xi\|\\
&\le (1-\delta)^k<\frac{\varepsilon}{2}.
}
Therefore $\|S_n^{\sigma}\xi\|<\varepsilon\,(n\ge n_k)$. Since $\varepsilon$ is arbitrary, we obtain $\disp \lim_{n\to \infty}\|S_n^{\sigma}\xi\|=0$ for every $\xi\in P^{\perp}(H)$, i.e., $H_{\sigma}=H$ holds. 
\end{proof}

\begin{proposition}\label{prop norm conv isolated}
Let $T_1\ge T_2\ge \cdots$ be a decreasing sequence of positive contractions on $H$. Assume that $1$ is isolated in $\sigma(T)$ and $\disp \lim_{n\to \infty}\|T_n-T\|=0$. Then $\disp \lim_{n\to \infty}\|S_n^{\sigma}-P\|=0$ for every $\sigma\in \mathscr{S}$. In particular, the generalized Paszkiewicz conjecture is true for $T_1\ge T_2\ge \cdots$. 
\end{proposition}
\begin{proof}By assumption, there exists $\delta\in (0,1)$ such that $\sigma(T)\cap (1-\delta,1)=\emptyset$.  
In particular, $\|T^k-P\|=\|T^kP^{\perp}\|\le (1-\delta)^k\xrightarrow{k\to \infty}0$. Therefore, for a given $\varepsilon>0$, there exists $k\in \N$ such that $\|T^k-P\|<\varepsilon$. 
Then by $T_jP=P\,(j\in \N)$ and $\disp \lim_{n\to \infty}\|T_{\sigma(n+1)}\dots T_{\sigma(n+k)}-T^k\|=0$, we have 
\eqa{
\limsup_{n\to \infty}\|(S_n^{\sigma})^*-P\|&=\limsup_{n\to \infty}\|(S_{n+k}^{\sigma})^*-P\|\\
&=\limsup_{n\to \infty}\|T_{\sigma(1)}\dots T_{\sigma(n)}(T_{\sigma(n+1)}\dots T_{\sigma(n+k)}-P)\|\\
&\le \limsup_{n\to \infty}\|T_{\sigma(n+1)}\dots T_{\sigma(n+k)}-P\|\\
&=\|T^k-P\|<\varepsilon.
}
Since $\varepsilon$ is arbitrary, we obtain $\disp \lim_{n\to \infty}\|(S_n^{\sigma})^*-P\|=\lim_{n\to \infty}\|S_n^{\sigma}-P\|=0$. 
\end{proof}
\begin{remark}\label{rem norm convergence and  1 isolated in spec(T)}
If $T_1\ge T_2\ge \dots$ is a decreasing sequence of positive contraction as in Proposition \ref{prop norm conv isolated}, then it has uniform spectral gap at 1 (this was pointed out by Hiroki Matui). 
Thus, Proposition \ref{prop norm conv isolated} is a special case of Proposition \ref{prop unif gap strong P}. On the other hand, note that for a general decreasing sequence of positive contractions with uniform spectral gap at 1, one cannot hope that $S_n$ converges in norm (consider e.g., the case where all $T_n$ are projections).   
\end{remark}
If $T_n$ converges to $T$ faster (but $1$ is possibly not isolated in $\sigma(T)$), then we also have a slightly weaker  conclusion. 
\begin{proposition}\label{prop norm summabilty implies Paszkiewicz}
Let $T_1\ge T_2\ge \cdots$ be a sequence of positive contractions on $H$. Assume that $\disp \sum_{n=1}^{\infty}\|T-T_n\|<\infty$ holds. Then for each $\sigma \in \mathscr{S}$ such that $\sup_{n\in \N}\sharp \sigma^{-1}(\{n\})<\infty$, $\disp \lim_{n\to \infty}S_n^{\sigma}=P$ {\rm (SOT)} holds. 
\end{proposition}
\begin{proof}
Let $\xi\in H$ be a unit vector and $\varepsilon>0$. By the hypothesis on $\sigma$, we have $\disp \sum_{n=1}^{\infty}\|T-T_{\sigma(n)}\|<\infty$. 
Choose $n_0\in \N$ such that $\disp \sum_{n=n_0+1}^{\infty}\|T-T_{\sigma(n)}\|<\frac{\varepsilon}{2}$ holds. 
Since $\disp \lim_{n\to \infty}T^n=P$ (SOT), there exists $n_1\in \N$ such that $\|T^nS_{n_0}^{\sigma}\xi-PS_{n_0}^{\sigma}\xi\|=\|T^nS_{n_0}^{\sigma}\xi-P\xi\|<\frac{\varepsilon}{2}$ for every $n\ge n_1$.

Let $S_{n,n_0+1}^{\sigma}=T_{\sigma(n+n_0)}\cdots T_{\sigma(n_0+1)}$ ($n$ products). 
For every $n\in \N$, we have 
\eqa{
    \|S_{n,n_0+1}^{\sigma}-T^n\|&\le \sum_{k=1}^{n-1}\|T_{\sigma(n+n_0)}\cdots T_{\sigma(n_0+k+1)}(T_{\sigma(n_0+k)}-T)T^{k-1}\|+\|(T_{\sigma(n+n_0)}-T)T^{n-1}\|\\
    &\le \sum_{k=n_0+1}^{\infty}\|T-T_{\sigma(k)}\|<\frac{\varepsilon}{2}.
}
Therefore, by $S_{n+n_0}^{\sigma}=S_{n,n_0+1}^{\sigma}S_{n_0}^{\sigma}$, it holds that for every 
$n\ge n_1$, 
\eqa{
    \|S_{n+n_0}^{\sigma}\xi-P\xi\|&\le \|(S_{n,n_0+1}^{\sigma}-T^n)S_{n_0}^{\sigma}\xi\|+\|T^nS_{n_0}^{\sigma}\xi-PS_{n_0}^{\sigma}\xi\|\\
    &<\|S_{n,n_0+1}^{\sigma}-T^n\|+\frac{\varepsilon}{2}<\varepsilon.
}
Since $\varepsilon$ is arbitrary, this shows that $\disp \lim_{n\to \infty}\|S_n^{\sigma}\xi-P\xi\|=0$.
\end{proof}

\section*{Acknowledgments}
We would like to thank Professors Masaru Nagisa and Mitsuru Uchiyama for useful discussions regarding the spectral order which we used in an earlier version of the paper and for pointing us to related literatures, Professor Hiroki Matui for his useful comment at Remark \ref{rem norm convergence and  1 isolated in spec(T)}. 
H.~Ando is supported by Japan Society for the Promotion of Sciences (JSPS) KAKENHI 20K03647. N.~Ozawa 
is partially supported by JSPS KAKENHI 20H00114 and  24K00527.

\bibliographystyle{siam}
\bibliography{references} 
\end{document}